\documentclass[5p]{elsarticle}

\usepackage{amsthm}
\usepackage{amsmath} 
\usepackage{amssymb}
\usepackage[table,dvipsnames]{xcolor}

\usepackage{enumitem}
\setlist[enumerate]{topsep=5pt,itemsep=-0.5ex,leftmargin=8mm}
\setlist[itemize]{leftmargin=*}
\usepackage{hyperref}

\newtheorem{theorem}{Theorem}[section]
\newtheorem{lemma}[theorem]{Lemma}
\newtheorem{proposition}[theorem]{Proposition}
\newtheorem{corollary}[theorem]{Corollary}

\theoremstyle{definition}
\newtheorem{definition}[theorem]{Definition}
\newtheorem{example}[theorem]{Example}

\theoremstyle{remark}
\newtheorem{remark}[theorem]{Remark}

\numberwithin{equation}{section}

\newcounter{syscounter}
\newenvironment{sysnum}{\begin{list}{($\Sigma{\arabic{syscounter}}$)}%
{\settowidth{\labelwidth}{($\Sigma4$)}
\settowidth{\leftmargin}{($\Sigma4$)~}%
\usecounter{syscounter}}}
{\end{list}}

\newcommand \N   {\mathbb{N}}
\newcommand \R   {\mathbb{R}}

\newcommand \A   {\mathcal{A}}

\newcommand \K   {\mathcal{K}}
\newcommand \Kinf{\mathcal{K_\infty}}

\newcommand \KL  {\mathcal{KL}}
\newcommand \LL  {\mathcal{L}}

\newcommand \Uc   {\mathcal{U}}
\newcommand \Dc   {\mathcal{D}}

\newcommand{\Rc}{\ensuremath{\mathcal{R}}}

\newcommand \srs   {\ \ \Rightarrow\ \ }

\newcommand \Iff   {\Leftrightarrow}

\newcommand \eps {\varepsilon}

\newtheorem*{Sol*}{Solution} 

\newif\ifAndo

\Andotrue

\journal{Systems \& Control Letters}

\begin{document}

\begin{frontmatter}

\title
{Uniform weak attractivity and criteria for practical global asymptotic stability}

\author[Pas]{Andrii Mironchenko}
\ead{andrii.mironchenko@uni-passau.de}

\address[Pas]{Faculty of Computer Science and Mathematics, University of Passau,
Innstra\ss e 33, 94032 Passau, Germany }


\begin{abstract}
A subset $\A$ of the state space is called uniformly globally weakly attractive if for any neighborhood $S$ of $\A$ and any bounded subset $B$ there is a uniform finite time $\tau$ so that any trajectory starting in $B$ intersects $S$ within the time not larger than $\tau$.
We show that practical uniform global asymptotic stability (pUGAS) is equivalent to the existence of a bounded uniformly globally weakly attractive set.
This result is valid for a wide class of distributed parameter systems, including time-delay systems, switched systems, many classes of PDEs and evolution differential equations in Banach spaces. We apply our results to show that existence of a non-coercive Lyapunov function ensures pUGAS for this class of systems. For ordinary differential equations with uniformly bounded disturbances, the concept of uniform weak attractivity is equivalent to the well-known notion of weak attractivity. It is however essentially stronger than weak attractivity for infinite-dimensional systems, even for linear ones.
\end{abstract}


\end{frontmatter}

\section{Introduction}
\label{sec:introduction}
%
%
%

A basic problem of control theory is to achieve a global stabilization of a nonlinear system. However, in many cases it is impossible (as in quantized control) or too costly to construct a feedback which lets all states of a system uniformly converge to the unique equilibrium point of a closed loop system. 
In such cases, it is natural to let the trajectories converge to a certain ball of a positive radius.

A basic concept which formalizes such a behavior is practical stability \cite{LLM90}.
This concept was extremely useful for stabilization of stochastic control systems \cite{ZhX13}, for control under quantization errors \cite{ShL12} and under communication bandwidth constraints \cite{WoB99}, etc. In \cite{JTP94,JMW96} conditions of small-gain type have been developed for stability of interconnections of input-to-state practically stable nonlinear systems.
In view of the importance of practical stability, it is essential to understand a relationship of practical stability to other stability notions. In this paper \textit{we prove a general criterion for practical uniform global asymptotic stability (pUGAS) for a wide class of control systems with disturbances} including evolution equations in Banach spaces, many classes of partial differential equations, switched systems and ordinary differential equations (ODEs). A key concept helping to achieve this aim is a \textit{uniform weak attractivity} \cite{MiW17a,MiW17b}.

A subset $\A$ of the state space $X$ is called \textit{globally weakly attractive} if for any neighborhood $S$ of $\A$ and any $x\in X$
the trajectory of a system, corresponding to initial condition $x$ intersects $S$.
The concept of weak attractivity has been introduced in \cite{Bha66} and played a prominent role in stability theory. In particular, this notion was one of the cornerstones for the development of stability theory in locally compact metric spaces in \cite{BhS02} and it was central for the proof of characterizations of input-to-state stability (ISS) of ODE systems in \cite{SoW95, SoW96}. 

In \cite{KaK15} conditions in the spirit of LaSalle's theorem which
guarantee that a set is a weak attractor are provided for finite-dimensional systems, without assuming forward completeness and boundedness of the set.


As useful the concept of weak attractivity is for ODEs and control systems over locally compact metric spaces, it is far too weak for general infinite-dimensional systems to imply any kind of uniform stability.
There are fundamental reasons for that: closed bounded balls are never compact in the norm topology
of infinite-dimensional normed linear spaces, nonuniformly globally
asymptotically stable nonlinear systems do not need to have bounded
finite-time reachability sets and even if they do, this still does not
guarantee uniform global stability \cite{MiW17b} (all such unpleasant things cannot occur in ODE setting).

In order to overcome these obstacles, in \cite{MiW17a,MiW17b} a notion of uniform weak attractivity has been introduced.
A subset $\A$ of the state space is called \textit{uniformly globally weakly attractive} if for any neighborhood $S$ of $\A$ and any bounded subset $B$ of $X$ there is a \textit{uniform} finite time $\tau$ so that any trajectory starting in $B$ intersects $S$ within time not larger than $\tau$.
It is crucial in this definition that $\tau$ depends on the set $B$ and does not depend on a particular $x\in B$. 

Uniform limit property (which is closely related to uniform weak attractivity) has already played a central role in the recent proof of characterizations of ISS for general infinite-dimensional control systems \cite{MiW17b} and in the theory of non-coercive Lyapunov functions for such systems \cite{MiW17a}. On the other hand, uniform weak attractivity was implicitly used in \cite[Proposition 2.2]{Tee13} and in \cite[Proposition 3]{SuT16} for characterizations of uniform global asymptotic stability of \textsl{finite-dimensional} hybrid and stochastic systems, and it is closely connected to the notion of recurrent sets \cite{BhS02,Kha11,TSS14}.

In this paper we use the concept of uniform weak attractivity to derive criteria for pUGAS of wide classes of distributed parameter systems. Our main technical result is that uniform global weak attractivity of a bounded set $\A$ together with a robust forward completeness implies the existence of a superset of $\A$ which is bounded, invariant and uniformly globally attractive. Using this fact \textit{we prove that existence of a bounded globally uniformly weakly attractive set is equivalent to practical UGAS of a system}.

Making a bridge to the framework of non-coercive Lyapunov functions, initiated in \cite{MiW17a}, \textit{we show that for a robustly forward complete system existence of a non-coercive Lyapunov function implies pUGAS.}

Although our main intention is to give novel criteria for practical UGAS, we show how one can modify the arguments to obtain the characterizations of UGAS.
{\it We specialize our results to linear infinite-dimensional systems and to ODEs with uniformly bounded disturbances. In the latter case, pUGAS is equivalent to the mere existence of a bounded globally weakly attractive set}.

\subsection{Notation}

The following notation will be used throughout these notes. Denote $\R_+:=[0,+\infty)$. 
For an arbitrary set $S$ and any $n\in\N$ the $n$-fold Cartesian product is $S^n:=S\times\ldots\times S$. 
The norm in a normed linear space $X$ is denoted by $\|\cdot\|$. The closure of a set $S \subset X$ w.r.t. norm $\|\cdot\|$ is denoted by $\overline{S}$.

For each $\A \subset X$ and any $x \in X$ we define the distance between $x$ and $\A$ by $\|x\|_{\A}:=\inf_{y\in \A}\|x-y\|$.
Define also $\|\A\|:=\sup_{x\in\A} \|x\|$.
The open ball in a normed linear space $X$ with radius $r$ around $\A \subset X$ is denoted by $B_r(\A):=\{x\in X \ :  \|x\|_{\A}<r\}$.
For $x\in X$ we denote $B_r(x) := B_r(\{x\})$ and $B_r:=B_r(0)$.

We use the comparison functions formalism:
\begin{equation*}
\begin{array}{ll}
{\K} &:= \left\{\gamma:\R_+\rightarrow\R_+:\ \gamma\mbox{ is continuous, strictly} \right.  \\
&\phantom{aaaaaaaaaaaaaaaaaaaa}\left. \mbox{ increasing and } \gamma(0)=0 \right\} \\
{\K_{\infty}}&:=\left\{\gamma\in\K:\ \gamma\mbox{ is unbounded}\right\}\\
{\LL}&:=\left\{\gamma:\R_+\rightarrow\R_+:\ \gamma\mbox{ is continuous and strictly}\right.\\
&\phantom{aaaaaaaaaaaaaaaa} \text{decreasing with } \lim\limits_{t\rightarrow\infty}\gamma(t)=0\}\\
{\KL} &:= \left\{\beta:\R_+\times\R_+\rightarrow\R_+:\ \beta \mbox{ is continuous,}\right.\\
&\phantom{aaaaaaaa}\left.\beta(\cdot,t)\in{\K},\ \beta(r,\cdot)\in {\LL},\ \forall t\geq 0,\ \forall r >0\right\} \\
\end{array}
\end{equation*}

\section{Preliminaries}
\label{sec:problem-statement}

The concept of a (time-invariant) system we understand in the following way:
\begin{definition}
\label{Steurungssystem}
Consider the triple $\Sigma=(X,\Dc,\phi)$, consisting of 
\begin{enumerate}[label=(\roman*)]
	\item A normed linear space $(X,\|\cdot\|)$, called the {state space}, endowed with the norm $\|\cdot\|$.
	\item A set of  disturbance values $D$, which is a nonempty subset of a certain normed linear space.
	\item A {space of disturbances} $\Dc \subset \{d:\R_+ \to D\}$
          satisfying the following two axioms.
\begin{itemize}
  \item \textit{The axiom of shift invariance} states that for all $d \in \Dc$ and all $\tau\geq0$ the time
shift $d(\cdot + \tau)$ is in $\Dc$.
  \item \textit{The axiom of concatenation} is defined by the requirement that for all $d_1,d_2 \in \Dc$ and for all $t>0$ the concatenation of $d_1$ and $d_2$ at time $t$
\begin{equation}
d(\tau) := 
\begin{cases}
d_1(\tau), & \text{ if } \tau \in [0,t], \\ 
d_2(\tau-t),  & \text{ otherwise},
\end{cases}
\label{eq:Composed_Input}
\end{equation}
belongs to $\Dc$.
\end{itemize}

	\item A transition map $\phi:\R_+ \times X \times \Dc \to X$.
\end{enumerate}
The triple $\Sigma$ is called a system, if the following properties hold:
\begin{sysnum}
	\item Forward completeness: for every $(x,d) \in X \times \Dc$ and
          for all $t \geq 0$ the value $\phi(t,x,d) \in X$ is well-defined.
	\item\label{axiom:Identity} The identity property: for every $(x,d) \in X \times \Dc$
          it holds that $\phi(0, x,d)=x$.
	\item Causality: for every $(t,x,d) \in \R_+ \times X \times
          \Dc$, for every $\tilde{d} \in \Dc$, such that $d(s) =
          \tilde{d}(s)$, $s \in [0,t]$ it holds that $\phi(t,x,d) = \phi(t,x,\tilde{d})$.
	\item \label{axiom:Continuity} Continuity: for each $(x,d) \in X \times \Dc$ the map $t \mapsto \phi(t,x,d)$ is continuous  on $\R_+$.
		\item \label{axiom:Cocycle} The cocycle property: for all $t,h \geq 0$, for all
                  $x \in X$, $d \in \Dc$ we have
$\phi(h,\phi(t,x,d),d(t+\cdot))=\phi(t+h,x,d)$.
\end{sysnum}
\end{definition}
Here  $\phi(t,x,d)$ denotes the state of a system $\Sigma=(X,\Dc,\phi)$ at the moment $t \in
\R_+$ corresponding to the initial condition $x \in X$ and the disturbance $d \in \Dc$.

Definition~\ref{Steurungssystem} covers a very broad class of systems, including differential equations in Banach spaces, time-delay systems, ODEs, many classes of partial differential equations etc.

\begin{remark}
\label{rem:axiom_of_concatenation}
The axiom of concatenation restricts the choice of a space of admissible inputs. Possible candidates are e.g. $L_p(\R_+,D)$, for $p\in[0,+\infty]$ or $PC(\R_+,D)$ (the class of piecewise-continuous functions which are right-continuous). At the same time, the space of continuous functions $C(\R_+,D)$ does not satisfy the axiom of concatenation (but is used in control theory much less frequently than $L_p$ and $PC$ spaces).
This axiom has been used only in Remark~\ref{rem:Invariance_Aeps} in order to prove invariance of reachability sets.
Hence, the results developed in this paper remain valid if instead of the axiom of concatenation we assume in advance that reachability sets are invariant.  
\end{remark}



%

For each $T\geq 0$, and each subset $S\subseteq X$, we define the sets of states which can be reached from $S$ under acting disturbances from $\Dc$ at time not exceeding $T$:
\[
\Rc^T(S):=\lbrace \phi(t,x,d):\ 0\leq t\leq T,\ d\in\Dc,\ x\in S\rbrace,
\]
\vspace{-7mm}
\begin{align*}
\Rc(S) :=\bigcup_{T\geq 0}\Rc^T(S) = \lbrace \phi(t,x,d):\ t\geq 0,\ d\in\Dc,\ x\in S\rbrace.
\end{align*}


Many classes of systems exhibit a stronger version of forward completeness, see \cite[p. 31]{KaJ11b}
\begin{definition}
\label{Def_RFC}
A system $\Sigma=(X,\Dc,\phi)$ is called robustly forward complete (RFC) 
if for any $C>0$ and any $\tau>0$ it holds that $\|\Rc^\tau(\overline{B_C})\| <\infty$.
\end{definition}

The next technical lemma will be useful in the sequel.	
\begin{lemma}
\label{lem:Norms_Changing}
For each $\A\subset X$ and $x\in X$ it holds that
\begin{eqnarray}
\|x\| - \|\A\| \leq \|x\|_{\A} \leq \|x\| + \|\A\|.
\label{eq:Relations_between_norms_1}
\end{eqnarray}
\vspace{-8mm}
\begin{eqnarray}
\|x\|_{\A} - \|\A\| \leq \|x\| \leq \|x\|_{\A} + \|\A\|.
\label{eq:Relations_between_norms_2}
\end{eqnarray}
\end{lemma}

\begin{proof}
For any $x,a\in X $ it holds that
\[
\|x\| \leq \|x-a\| + \|a\| \leq \|x-a\| + \|\A\|.
\]	
Taking $\inf_{a\in \A}$, we obtain $\|x\| {-} \|\A\| \leq \inf_{a\in \A} \|x-a\| {=} \|x\|_{\A}.$
Computing $\inf_{a\in \A}$ in the inequality $\|x-a\| \leq \|x\| + \|a\|_\A$ we show \eqref{eq:Relations_between_norms_1}. Inequality \eqref{eq:Relations_between_norms_2} follows from \eqref{eq:Relations_between_norms_1}.
\end{proof}

\begin{proposition}
\label{prop:RFC_equivalence_wrt_A}
Let $\Sigma$ be a system. $\Sigma$ is RFC if and only if there is a bounded set $\A \subset X$ so that 
\begin{eqnarray*}
\hspace{-4mm} C>0,\ \tau>0 \srs \sup_{\|x\|_{\A}\leq C,\ d \in \Dc,\ t \in [0,\tau]}\hspace{-7mm}\|\phi(t,x,d)\|_{\A} < \infty.
\end{eqnarray*}
\end{proposition}

\begin{proof}
Let $\Sigma$ be an RFC system. By means of Lemma~\ref{lem:Norms_Changing} we obtain
\begin{multline*}
\sup_{\|x\|_{\A}\leq C,\ d \in \Dc,\ t \in [0,\tau]}\hspace{-7mm}\|\phi(t,x,d)\|_{\A} \\
\leq \sup_{\|x\|\leq C + \|\A\|,\ d \in \Dc,\ t \in [0,\tau]}\hspace{-7mm}\big(\|\phi(t,x,d)\| + \|\A\|\big) < \infty.
\end{multline*}
%
%
The proof of a converse implication is analogous.
\end{proof}

\begin{definition}
\label{def:0-invariance}
Let $\Sigma=(X,\Dc,\phi)$ be a system.
\begin{enumerate}[label=(\roman*)]
	\item A set $\A \subset X$ is called invariant if $\Rc(\A) \subset \A$.
	\item An invariant set $\A \subset X$ is called robustly invariant if $\forall \eps >0$, $\forall h>0$  $\exists \delta = \delta (\eps,h)>0$, so that:
\begin{eqnarray}
\hspace{-15mm} t\in[0,h],\ \|x\|_{\A} \leq \delta,\ d\in\Dc \; \Rightarrow \;  \|\phi(t,x,d)\|_{\A} \leq \eps.
\label{eq:RobEqPoint}
\end{eqnarray}
	\item A robust invariant point $x_0 \in X$ is called a robust equilibrium.
\end{enumerate}
\end{definition}

\begin{remark}
\label{rem:Robustness_of_Equilibirum}
  Robustness of a zero equilibrium is equivalent to continuity of a function $x\mapsto \sup_{t\in[0,h],\ d\in \Dc}\|\phi(t,x,d)\|$ at $x=0$, for all $h>0$. An example of a system with a non-robust equilibrium is given in Example~\ref{exam:UGATT_Robustness_Invariant_Sets}.
\end{remark}

We will investigate the following stability properties:

\begin{definition}{}
    \label{d:stability_new}
Consider a system $\Sigma=(X,\Dc,\phi)$. A set $\A \subset X$ is called
\begin{enumerate}[label=(\roman*)]
  \item  Lagrange stable, if there are $\sigma \in \Kinf$ and $c>0$:
\[
\hspace{-5mm} x\in X,\ t \geq 0,\ d\in\Dc \srs  \|\phi(t,x,d)\|_{\A} \leq \sigma(\|x\|_{\A}) + c.
\]
  \item uniformly (locally) stable (ULS), if $\forall \eps >0$ \ $\exists \delta>0$:
\begin{eqnarray}
\hspace{-10mm} \|x\|_\A \leq \delta,\ d\in\Dc,\ t \geq 0 \srs\|\phi(t,x,d)\|_\A \leq \eps.
\label{eq:Uniform_Stability}
\end{eqnarray}

\item uniformly globally stable (UGS), if it is Lagrange stable with $c=0$.

\item uniformly globally asymptotically stable (UGAS)
if there is a $\beta \in \KL$ s.t. for all $x \in X$, $d\in \Dc$ and $t \geq 0$
\begin{equation}
\label{eq:UGAS_wrt_A_estimate}
\hspace{-1mm}\|\phi(t,x,d)\|_\A \leq \beta(\|x\|_\A,t).
\end{equation}

\item practically uniformly globally asymptotically stable (pUGAS) 
if there are $\beta \in \KL$ and $c>0$ such that for all $x \in X$, $d\in \Dc$ and $t \geq 0$
\begin{equation}
\label{eq:practical_UGAS_wrt_A_estimate}
\hspace{-1mm} \|\phi(t,x,d)\|_\A \leq \beta(\|x\|_\A,t) +c.
\end{equation}

\item globally weakly attractive, if
\begin{eqnarray}
\hspace{-12mm} x\in X,\ d \in\Dc,\ \eps>0 \srs \exists t: \ \|\phi(t,x,d)\|_{\A} \leq \eps.
\label{eq:LIM_inf_form}
\end{eqnarray}

\item \label{def:UniformGlobalWeakAttractivity} uniformly globally weakly attractive, if for every $\eps>0$ and
for every $r>0$ there is a $\tau = \tau(\eps,r)$ such that 
\begin{eqnarray}
\hspace{-13mm} \|x\|_{\A}\leq r,\ d\in\Dc \ \Rightarrow \ \exists t  \leq \tau:\ \|\phi(t,x,d)\|_{\A} \leq \eps.
\label{eq:Uniform_weak2}
\end{eqnarray}
\item \label{def:UniformGlobalAttractivity} uniformly globally attractive (UGATT), if for any $r,\eps >0$ there exists $\tau=\tau(r,\eps)$ so that
\begin{eqnarray}
\hspace{-16mm} \|x\|_{\A} \leq r,\ d\in\Dc,\ t \geq \tau(r,\eps) \ \Rightarrow \ \|\phi(t,x,d)\|_{\A} \leq \eps.
\label{eq:UAG_with_zero_gain}
\end{eqnarray}
\end{enumerate}
\end{definition}

\begin{remark}
\label{rem:UGS_equals_LagrStab_and_ULS}
For any system and any $\A\subset X$ the following holds:
\begin{center}
$\A$ is UGS \quad$\Iff$\quad $\A$ is ULS $\ \wedge\ $ $\A$ is Lagrange stable,
\end{center}
see \cite[Lemma 1.2]{SoW96} for a similar result.
\end{remark}

\begin{remark}
\label{rem:Invariance_of_pUGAS_Lagrange_Stability_wrt_choice_of_aboundedset}
  Note that a bounded set $\A\subset X$ is Lagrange stable if and only if $\{0\}$ is Lagrange stable (which easily follows from Lemma~\ref{lem:Norms_Changing}).
 Similarly $\A\subset X$ is pUGAS iff $\{0\}$ is pUGAS. In view of these facts \textit{we will call a system $\Sigma$ Lagrange stable or pUGAS respectively if there is a bounded subset of $X$ which possesses such properties}.
\end{remark}

\begin{definition}
\label{def:Unif_Ultimate_Boundedness}
A system $\Sigma=(X,\Uc,\phi)$ is called uniformly ultimately bounded, if there is $K>0$ so that for each $r>0$ there is $T=T(r)$:
\begin{eqnarray}
\hspace{-5mm} \|x\| \leq r,\ d \in\Dc,\ t\geq T \srs  \|\phi(t,x,d)\|\leq K.
\label{eq:Unif_Ult_Boundedness}
\end{eqnarray}  
\end{definition}

\begin{remark}
\label{rem:UGATT_implies_Unif_Boundedness}
Let $\Sigma$ be any system. Existence of a bounded UGATT set implies uniform ultimate boundedness of $\Sigma$.
Indeed, let $\A$ be a bounded invariant UGATT set. Pick $\eps:=1$ 
and set $T(r):=\tau(r,1)$. Then
\begin{eqnarray*}
\hspace{-8mm} \|x\|_\A \leq r,\ d \in\Dc,\ t\geq T(r) \srs  \|\phi(t,x,d)\|\leq \|\A\| + 1,
\end{eqnarray*}
and Lemma~\ref{lem:Norms_Changing} ensures uniform ultimate boundedness \text{of $\Sigma$} as:
\begin{eqnarray*}
\hspace{-9mm} \|x\| \leq r,\ d \in\Dc,\ t\geq T(r{+}\|\A\|) \srs  \|\phi(t,x,d)\|\leq \|\A\| {+} 1.
\end{eqnarray*}
\end{remark}

%
%
%
%
%
%

Let a system $\Sigma$ be fixed and let $V:X \to\R$ be continuous.
Given $x\in X$ and $d\in \mathcal{D}$, we consider the (right-hand lower) Dini
derivative of the continuous function $t \mapsto V(\phi(t,x,d))$ at $t=0$:
\begin{equation}
\label{UGAS_wrt_D_LyapAbleitung}
\dot{V}_d(x):=\mathop{\underline{\lim}} \limits_{t \rightarrow +0} {\frac{1}{t}\big(V(\phi(t,x,d))-V(x)\big) }.
\end{equation}
{\nolinebreak We call this the Dini derivative of $V$ along trajectories \text{of $\Sigma$}.}

Recently in \cite{MiW17a} a novel concept of a non-coercive Lyapunov function has been introduced.
\begin{definition}
\label{def:UGAS_LF_With_Disturbances}
Let $\A \subset X$ be a bounded set.
A continuous function $V:X \to \R_+$ is called a \textit{non-coercive Lyapunov function (or functional)} for a system $\Sigma=(X,\Dc,\phi)$ w.r.t. $\A$,  if there exist
$\psi_2 \in \Kinf$ and $\alpha \in \K$
such that $V(x)=0$ for all $x\in \A$, the inequalities 
\begin{equation}
    \label{eq:1}
    0 < V(x) \leq \psi_2(\|x\|_{\A}) \quad \forall x \in X \backslash \A
\end{equation}
hold
and the Dini derivative of $V$  along trajectories of $\Sigma$ satisfies 
\begin{equation}
\label{DissipationIneq_UGAS_With_Disturbances}
\dot{V}_d(x) \leq -\alpha(\|x\|_{\A})
\end{equation}
for all $x \in X\backslash\A$ and all $d \in \Dc$. 
\end{definition}

The importance of non-coercive Lyapunov functions is due to the following result, see \cite{MiW17a}:

\begin{proposition}
    \label{t:noncoeLT}
    Consider an RFC system $\Sigma=(X,\Dc,\phi)$	and assume that $\{0\}$ is its robust equilibrium.
		If $V$ is a non-coercive Lyapunov function for $\Sigma$ w.r.t $\{0\}$, then $\{0\}$ is UGAS.
\end{proposition}

Non-coercive Lyapunov functions are easier to construct than coercive ones.
For example, construction of Lyapunov functions for linear systems over Banach spaces by means of the Lyapunov equation provides in general merely non-coercive Lyapunov functions, \cite[Theorem 5.1.3]{CuZ95}.

We will use non-coercive Lyapunov functions to derive a sufficient condition for pUGAS of systems with disturbances.

\section{Characterization of pUGAS}
\label{sec:pUGAS_equals_ULIM}

In this section we characterize pUGAS as follows:
\begin{theorem}
\label{thm:RFC_weak_attr_imply_pUGAS}
Let $\Sigma$ be a system. The following statements are equivalent:
\begin{enumerate}[label=(\roman*)]
	\item\label{enum:pUGAS_item1} $\Sigma$ is pUGAS
	\item\label{enum:pUGAS_item11} $\Sigma$ is Lagrange stable and there is a bounded invariant UGATT set.	
	\item\label{enum:pUGAS_item2} $\Sigma$ is RFC and there is a bounded invariant UGATT set.	
	\item\label{enum:pUGAS_item21} $\Sigma$ is RFC and uniformly ultimately bounded.
	\item\label{enum:pUGAS_item3} $\Sigma$ is RFC and there is a bounded uniformly globally weakly attractive set.
\end{enumerate}
\end{theorem}

We start with
\begin{definition}
\label{def:first_prolongation}
For each $\eps>0$ and any $\A\neq \emptyset$ define 
\begin{eqnarray}
\A_\eps:=\Rc\big(B_\eps(\A)\big),\qquad P_+(\A):=\bigcap_{\eps>0} \A_\eps. 
\label{eq:A_eps_definition}
\end{eqnarray}
\end{definition}

\begin{remark}
\label{rem:Invariance_Aeps}
If $\A\neq\emptyset$, then for any $\eps>0$ the set $\A_\eps$ is invariant. Indeed, pick any $x\in\A_\eps$, $t\geq0$ and $d\in\Dc$.
Then there are $t_1\geq0$, $x_1\in\A$ and $d_1\in\Dc$ so that $x=\phi(t_1,x_1,d_1)$.
Due to the cocycle axiom ($\Sigma$\ref{axiom:Cocycle}) we have that $\phi(t,x,d)=\phi(t+t_1,x_1,d_2)$, where $d_2$ is a concatenation of $d_1$ and $d$ (which belongs to $\Dc$ in view of the axiom of concatenation). 
Due to the axiom of shift-invariance $d_2\in \Dc$, and thus $\phi(t,x,d) \in\A_\eps$.
\end{remark}

We start by showing invariance of $P_+(\A)$.
%

\begin{lemma}
\label{lem:properties_A_epsilon_0}
 Let a system $\Sigma$ be given and let $\A \subset X$, $\A \neq \emptyset$. Then $P_+(\A)$ is a nonempty invariant set.
\end{lemma}

\begin{proof} Since $\A \subset \A_\eps$ for any $\eps>0$, we have $\A \subset P_+(\A)$, and hence $P_+(\A)$ is nonempty.
For any $\eps>0$ the set $\A_\eps$ is invariant by Remark~\ref{rem:Invariance_Aeps}.

Pick any $x\in \bigcap_{\eps>0} \A_\eps$. Since $\A_\eps$ is invariant for all $\eps>0$, for any $t\geq 0$ and any $d\in\Dc$ it holds that $\phi(t,x,d) \in \A_\eps$, $\eps>0$. Thus, $\phi(t,x,d) \in \bigcap_{\eps>0} \A_\eps=P_+(\A)$, which shows invariance of $P_+(\A)$.
\end{proof}

%
%
%
%
%

Next  we present our main technical result, showing that for any bounded uniformly globally weakly attractive set there is an  invariant superset with a much stronger UGATT property.
\begin{proposition}
\label{prop:from_ULIM_to_UGATT}
Assume that $\Sigma$ is an RFC system and that $\A$ is a bounded uniform global weak attractor 
for $\Sigma$. Then for any $\eps>0$ the set $\A_\eps$ is bounded, invariant and UGATT.
%
\end{proposition}

\begin{proof}
Pick any $\eps>0$, any $R>0$ and fix them. Since $\A$ is a global uniform weak attractor of an RFC system $\Sigma$, there exists $\tilde\tau=\tilde\tau(\eps, R)>0$ so that 
\begin{eqnarray}
\hspace{-8mm} \|x\|_{\A} \leq R,\ d\in\Dc \srs \exists \bar{t}\leq \tilde\tau:\ \|\phi(\bar{t},x,d)\|_{\A} \leq \eps.
\label{eq:ULIM_UGATT_Section0}
\end{eqnarray}
Without loss of generality assume that $\tilde\tau$ is decreasing w.r.t. $\eps$ and increasing w.r.t. $R$.
Then $\tilde\tau$ is integrable on compact sets.
Define
\begin{eqnarray*}
\tau(\eps,R):= \frac{2}{\eps R}\int_R^{2R}\int_{\eps/2}^{\eps} \tilde\tau(\eps_1,R_1)d\eps_1dR_1.
\end{eqnarray*}
Clearly, $\tau$ is increasing w.r.t. $R$, decreasing w.r.t. $\eps$, continuous and 
$\tau(\eps,R)\geq\tilde\tau(\eps,R)$ for any $\eps,R>0$. Thus,
\begin{eqnarray}
\hspace{-8mm} \|x\|_{\A} \leq R,\ d\in\Dc \srs \exists \bar{t}\leq \tau:\ \|\phi(\bar{t},x,d)\|_{\A} \leq \eps.
\label{eq:ULIM_UGATT_Section1}
\end{eqnarray}

In particular, since $\tau$ is continuous, for any $\eps>0$ and $d\in\Dc$ it holds that
\begin{eqnarray}
\hspace{-8mm} \|x\|_{\A} \leq \eps  \srs \exists \tilde{t} \in(0,\tau(\eps, \eps)):\ \|\phi(\tilde{t},x,d)\|_{\A} \leq \eps.
\label{eq:ULIM_UGATT_Section2}
\end{eqnarray}
This means, that trajectories, emanating from $B_{\eps}(\A)$ return to this ball in time not larger than $\tau(\eps, \eps)$.
Hence $\A_\eps =\Rc^{\tau(\eps,\eps)}\big(B_\eps(\A)\big)$.
As $\Sigma$ is RFC, $\A_\eps$ is bounded for any $\eps>0$.

Now the cocycle property implies for any $t>\bar{t}$ that
\begin{eqnarray*}
\phi(t,x,d) = \phi\big(t-\bar{t},\phi(\bar{t},x,d),d(\cdot + \bar{t})\big).
\end{eqnarray*}
This ensures that 
\begin{eqnarray*}
\|x\|_{\A} \leq R ,\ d\in\Dc,\ t\geq \bar{t} \srs \phi(t,x,d)  \in \A_\eps,
\end{eqnarray*}
which means that for each $d\in\Dc$
\begin{eqnarray}
\hspace{-5mm} \|x\|_{\A} \leq R,\ t\geq \tau(R,\eps) \srs \|\phi(t,x,d)\|_{\A_\eps} = 0.
\label{eq:UGATT_wrt_Aeps}
\end{eqnarray}
Since $\A \subset \A_\eps$ for any $\eps>0$ it follows that 
$\|x\|_{\A} \leq \|x\|_{\A_\eps} + C$, where
$C:=\sup_{y\in\partial\A_\eps,\ z\in\partial\A,\ [y,z]\bigcap\A=\emptyset}\|y-z\|$ and
$[y,z]:=\{ty+(1-t)z:t\in(0,1)\}$.

Define $\tilde\tau(R,\eps):=\tau(R+C,\eps)$. We have:
\begin{eqnarray*}
\hspace{-9mm} \|x\|_{\A_\eps} \leq R ,\ d\in \Dc ,\ t\geq \tilde\tau(R,\eps) \srs \|\phi(t,x,d)\|_{\A_\eps} =0,
\end{eqnarray*}
which shows that $\A_\eps$ is UGATT.
\end{proof}

We call a function $h: \R_+^2 \to \R_+$ increasing, if $(r_1,R_1) \leq (r_2,R_2)$
implies that $h(r_1,R_1) \leq h(r_2,R_2)$, where we use the component-wise
partial order on $\R_+^2$. We call $h$ strictly increasing if $(r_1,R_1)
\leq (r_2,R_2)$ and $(r_1,R_1) \neq (r_2,R_2)$ imply $h(r_1,R_1) <
h(r_2,R_2)$.

We need the following characterization of RFC in terms of comparison-like functions, shown in \cite{MiW17a}:
\begin{lemma}
\label{lem:RFC_criterion}
Consider a  forward complete system $\Sigma=(X,\Dc,\phi)$. The following statements are equivalent:
\begin{enumerate}
	\item[(i)] $\Sigma$ is RFC.
	\item[(ii)] there exists a continuous, increasing function $\mu: \R_+^2 \to \R_+$, such that
 \begin{equation}
    \label{eq:8}
\hspace{-11mm}    x\in X,\ d\in \Dc,\ t \geq 0 \srs \| \phi(t,x,d) \| \leq \mu( \|x\|,t).
\end{equation}	
\end{enumerate}
\end{lemma}

The next lemma ensures that existence of a bounded uniform globally weakly attractive set for an RFC system implies its Lagrange stability. It is close in spirit to \cite[Proposition 7]{MiW17b}.
\begin{lemma}
\label{lem:Weak_attr_plus_RFC_imply_pUGS}
Let $\Sigma$ be an RFC system. If there is a bounded globally uniformly weakly attractive set $\A \subset X$, then $\Sigma$ is Lagrange stable.
\end{lemma}

%

\begin{proof}
Pick any $r>0$ and set $\eps:=\frac{r}{2}$.
Since there is a set $\A\subset X$ which is uniformly globally weakly attractive, there exists a $\tau=\tau(r)$ so that 
\begin{eqnarray}
\hspace{-9mm}\|x\|_{\A} \leq r,\ d\in\Dc \srs \exists t\leq \tau(r):\ \|\phi(t,x,d)\|_\A \leq \eps.
\label{eq:ULIM_pGS_theorem1_v2}
\end{eqnarray}
Without loss of generality we can assume that $\tau$ is increasing in $r$, in particular, it is locally integrable.
Defining $\bar\tau(r):=\frac{1}{r}\int_r^{2r}\tau(s)ds$ we see that $\bar\tau(r) \geq \tau(r)$ and $\bar\tau$ is continuous.

Since $\Sigma$ is RFC, by Lemma~\ref{lem:RFC_criterion} there exists a continuous, increasing function 
$\mu: \R_+^2 \to \R_+$, such that for all $x\in X, d\in \Dc$ and all $t \geq 0$ we have \eqref{eq:8}.
In particular,
\begin{eqnarray}
\hspace{-9mm}\|x\|_\A {\leq} r,\, d{\in}\Dc,\, t{\in}[0,\tau(r)] \ \Rightarrow\  \|\phi(t{,}x{,}d)\|_\A \leq \mu(r,\bar\tau(r)).
\label{eq:ULIM_pGS_theorem2_v2}
\end{eqnarray}
Since $\mu$ and $\bar\tau$ are both continuous and increasing, $\tilde\sigma:r\mapsto \mu(r,\bar\tau(r))$ is again continuous and increasing. Also from \eqref{eq:ULIM_pGS_theorem1_v2} and \eqref{eq:ULIM_pGS_theorem2_v2} it is clear that $\tilde\sigma(r) \geq \frac{r}{2}$ for any $r>0$.

Seeking a contradiction, assume that there exist $x\in B_r(\A)$, $d\in\Dc$ and $t\geq0$ so that $\|\phi(t,x,d)\|_\A  > \tilde
\sigma(r)$. Define $t_m:=\sup\{s \in[0,t]: \|\phi(s,x,d)\|_\A \leq r\}$ which is well-defined, since ($\Sigma$\ref{axiom:Identity}) implies that $\|\phi(0,x,d)\|_\A = \|x\|_\A \leq r$.

In view of the cocycle property, it holds that 
\[
\phi(t,x,d) = \phi\big(t-t_m,\phi(t_m,x,d),d(\cdot + t_m)\big),
\]
and $d(\cdot + t_m) \in\Dc$ due to the axiom of shift invariance.

Assume that $t-t_m \leq\tau(r)$. Since $\|\phi(t_m,x,d)\|_\A\leq r$,
\eqref{eq:ULIM_pGS_theorem2_v2} implies that $\|\phi(s,x,d)\|_\A \leq \tilde \sigma(r)$ for all $s \in [t_m,t]$.
Otherwise, if $t-t_m >\tau(r)$, we may apply
\eqref{eq:ULIM_pGS_theorem1_v2} and there exists $t^*<\tau(r)$, so that 
\[
\|\phi\big(t^*,\phi(t_m,x,d),d(\cdot + t_m)\big)\|_\A =
\|\phi(t^*+t_m,x,d)\|_\A \leq \eps = \frac{r}{2}.
\] 
This contradicts the definition of $t_m$, since $t_m+t^*<t$.
Hence
\begin{eqnarray}
\hspace{-9mm}\|x\|_\A \leq r,\ d\in\Dc,\ t\geq 0 \srs  \|\phi(t,x,d)\|_\A \leq \tilde\sigma(r).
\label{eq:ULIM_pGS_theorem3_v2}
\end{eqnarray}
Denote $\sigma(r):=\tilde\sigma(r) - \tilde\sigma(0)$, for any $r\geq 0$. Clearly, $\sigma\in\Kinf$.
Then from \eqref{eq:ULIM_pGS_theorem3_v2} we have for all $x\in X,\ d\in\Dc,\
t\geq 0$ that
\begin{eqnarray}
 \|\phi(t,x,d)\|_\A \leq \sigma(\|x\|_\A) + \tilde\sigma(0),
\label{eq:ULIM_pGS_theorem4_v2}
\end{eqnarray}
which shows Lagrange stability of $\Sigma$.
\end{proof}

The following lemma is routine:
\begin{lemma}
\label{lem:UGATT_plus_LagrangeSt_imply_pUGAS}
Let $\Sigma$ be a system and let $\A \subset X$ be a bounded invariant set of $\Sigma$. Then:
\begin{enumerate}[label=(\roman*)]
   \item If $\A$ is UGATT and Lagrange stable, then $\Sigma$ is pUGAS.
   \item If $\A$ is UGATT and globally stable, then $\Sigma$ is UGAS.
\end{enumerate}
\end{lemma}

\begin{proof}
Let us start with (i).
Since $\A$ is Lagrange stable, there is a $\sigma\in\Kinf$ and $c>0$:
\begin{eqnarray}
\hspace{-9mm}t\geq 0,\, \|x\|_\A \leq \delta,\, d \in \Dc \srs \|\phi(t,x,d)\|_\A \leq \sigma(\delta) + c.
\label{eq:TmpEstimate}
\end{eqnarray}
Define $\eps_n:=\frac{1}{2^n} \sigma(\delta)$, for $n \in \N$. 
UGATT of $\A$ ensures that there is a sequence of times $\tau_n:=\tau(\eps_n,\delta)$ which we assume without loss of generality to be strictly increasing, such that 
\[
t \geq \tau_n,\ \|x\|_\A \leq \delta,\ d \in \Dc \srs \|\phi(t,x,d)\|_\A \leq \eps_n.
\]
Define $\omega(\delta,\tau_n):=\eps_{n-1}$, for $n \in \N$, $n \neq 0$ and extend the function $\omega(\delta,\cdot)$ for $t \in \R_+ \backslash \{\tau_n: n \in \N\}$ so that $\omega(\delta,\cdot) \in \LL$. 
Note that for any $n\in\N$ and for all $t \in (\tau_n,\tau_{n+1})$ it holds that $\|\phi(t,x,d)\|_\A \leq \eps_n < \omega(\delta,t)$.
Doing this for all $\delta \in \R_+$ we obtain function $\omega$.

Now choose $\tilde\beta(r,t)=\sup_{0 \leq s \leq r}\omega(s,t) \geq \omega(r,t)$. Obviously, $\tilde\beta$ is non-decreasing w.r.t. the first argument and decreasing w.r.t. the second one. It can be estimated from above by $\beta \in\KL$ and 
\eqref{eq:practical_UGAS_wrt_A_estimate} is satisfied with such a $\beta$. 

Let us show (ii). Since $\A$ is globally stable, \eqref{eq:TmpEstimate} holds with $c=0$. The same argument as above shows now UGAS of $\A$.
\end{proof}

Finally, we can prove Theorem~\ref{thm:RFC_weak_attr_imply_pUGAS}.
\begin{proof} (of Theorem~\ref{thm:RFC_weak_attr_imply_pUGAS}.)
\ref{enum:pUGAS_item1} $\Rightarrow$ \ref{enum:pUGAS_item11}. 
Assume, that $\Sigma$ is a pUGAS system.
Then there is a $\beta\in\KL$ and $c>0$ so that for all $x\in X$, $t\geq 0$ and $d\in\Dc$ we have
\begin{equation}
\label{eq:practical_UGAS_Xnorm}
\hspace{-1mm}\|\phi(t,x,d)\| \leq \beta(\|x\|,t) +c.
\end{equation}

Then $\overline{B_c(0)}$ is uniformly globally attractive (but may be not invariant).
According to Proposition~\ref{prop:from_ULIM_to_UGATT} there is a bounded invariant UGATT set $\A$ (which can be chosen e.g. as $(B_c(0))_\eps$ for any $\eps>0$).

At the same time \eqref{eq:practical_UGAS_Xnorm} implies Lagrange stability since $\beta(\|x\|,t)\leq \beta(\|x\|,0)$ for any $x \in X,\ t\geq0$.

\ref{enum:pUGAS_item11} $\Rightarrow$ \ref{enum:pUGAS_item2}. Clear.

\ref{enum:pUGAS_item2} $\Rightarrow$ \ref{enum:pUGAS_item21}. Follows by Remark~\ref{rem:UGATT_implies_Unif_Boundedness}.

\ref{enum:pUGAS_item21} $\Rightarrow$ \ref{enum:pUGAS_item3}. Evident, since the set $B_K$ in Definition~\ref{def:Unif_Ultimate_Boundedness}
 is uniformly globally weakly attractive.

\ref{enum:pUGAS_item3} $\Rightarrow$ \ref{enum:pUGAS_item1}. 
Let $\Sigma$ be an RFC system and there is a bounded uniformly globally weakly attractive set for $\Sigma$.
By Proposition~\ref{prop:from_ULIM_to_UGATT} there is a bounded invariant UGATT set for $\Sigma$.
Lemma~\ref{lem:Weak_attr_plus_RFC_imply_pUGS} implies that $\Sigma$ is Lagrange stable and finally 
Lemma~\ref{lem:UGATT_plus_LagrangeSt_imply_pUGAS} proves that $\Sigma$ is pUGAS.
\end{proof}

\begin{remark}
\label{rem:Forward_completeness_is_not_enough}
Theorem~\ref{thm:RFC_weak_attr_imply_pUGAS} does not hold, if we require instead of RFC merely forward completeness, see an example in \cite{MiW17a}.
\end{remark}

\subsection{Existence of a non-coercive Lyapunov function implies pUGAS}

As an application of Theorem~\ref{thm:RFC_weak_attr_imply_pUGAS} we obtain the following non-coercive Lyapunov sufficient condition for pUGAS of RFC systems.

\begin{corollary}
Let $\Sigma=(X,\Dc,\phi)$ be an RFC system. 
If there is a bounded set $\A\subset X$ so that $\Sigma$ possesses a non-coercive Lyapunov function w.r.t. $\A$, then $\Sigma$ is practically UGAS.
\end{corollary}

\begin{proof}
Follows by Theorem~\ref{thm:RFC_weak_attr_imply_pUGAS} and Proposition~\ref{prop:noncoeLT_plus_FC_implies_ULIM_v2} which is proved next.
\end{proof}


The following relation between non-coercive Lyapunov functions and uniform weak attractivity is valid:
\begin{proposition}
    \label{prop:noncoeLT_plus_FC_implies_ULIM_v2}
    Consider a forward complete system $\Sigma=(X,\Dc,\phi)$ and let $\A \subset X$ be bounded. If there
    exists a non-coercive Lyapunov function $V$ for $\A$, then $\A$ is uniformly globally weakly attractive.
\end{proposition}

This proposition has been shown in \cite{MiW17a} for $\A=\{0\}$. The proof is completely analogous for general bounded $\A$.
We include it here merely for the sake of completeness.

\begin{proof}
Let $V$ be a non-coercive Lyapunov function and let $\alpha \in \mathcal{K}$ be such that the decay estimate \eqref{DissipationIneq_UGAS_With_Disturbances} is true.
Along any trajectory $\phi$ of a forward complete system $\Sigma$ we have the inequality
\begin{equation}
\dot V_{d(t+\cdot)}(\phi(t,x,d)) \leq -\alpha(\|\phi(t,x,d)\|_{\A}), \quad
\forall t\geq 0.
\label{eq:JustEq1}
\end{equation}
Since $\alpha \in\K$, it follows that
\begin{equation}
V(\phi(t,x,d)) - V(x) \leq -\int_0^t \alpha(\|\phi(s,x,d)\|_{\A})ds,
\label{eq:JustEq3}
\end{equation}
which implies that for all $t\geq 0$ we have
\begin{eqnarray}
\int_0^t \alpha(\|\phi(s,x,d)\|_{\A})ds  
\leq
V(x)  \leq \psi_2(\|x\|_{\A}).
\label{eq:Vintbound_2_v2}
\end{eqnarray}

%
Fix $\eps>0$ and $r>0$ and define $\tau(\eps,r):=\frac{\psi_2(r) +
  1}{\alpha(\eps)}$. We claim that this choice of $\tau = \tau(\eps,r)$
yields an appropriate time to conclude uniform weak attractivity of
$\A$. Assume to the contrary that there exist $x\in B_r(\A)$ and $d\in \Dc$ with the property
$\|\phi(t,x,d)\|_{\A} \geq \eps$ for all $t\in[0,\tau(\eps,r)]$. In view of \eqref{eq:Vintbound_2_v2} and since $\alpha\in\K$, we obtain
\[
\psi_2(r) + 1 = \tau(\eps,r) \alpha(\eps) \leq  \int_0^{\tau(\eps,r)}\hspace{-4mm} \alpha(\|\phi(s,x,d)\|_{\A})ds \leq \psi_2(r),
\]
a contradiction.
\end{proof}

\subsection{Strengthening of Proposition~\ref{prop:from_ULIM_to_UGATT}}

Proposition~\ref{prop:from_ULIM_to_UGATT} shows the existence of bounded UGATT sets for RFC systems possessing bounded uniform global attractors. A mere existence result was satisfactory for characterizing pUGAS systems, but for other purposes it is of interest to find tight estimates of the attraction region. 

Next proposition, motivated by \cite[Theorem 1.25 in Chapter V]{BhS02}, shows under additional assumptions a result of such kind.
\begin{proposition}
\label{prop:from_ULIM_to_UGATT_strengthening}
Let $\Sigma$ be an RFC system and assume that $\A$ is a bounded  uniform global weak attractor 
for $\Sigma$. If
\begin{eqnarray}
\forall\eps>0\  \exists \delta >0:\ \sup_{y\in \A_\delta}\|y\|_{P_+(\A)}<\eps,
\label{eq:Continuity_reachability_sets}
\end{eqnarray}
then $P_+(\A)$ is a bounded invariant UGATT set.

If additionally $\A$ is a robust invariant set, then there is no UGATT set $S$ so that $\A\subset S \subset P_+(\A)$ and $S$ is not dense in $P_+(\A)$ w.r.t. topology defined by $\|\cdot\|$.
\end{proposition}

\begin{proof}
Let \eqref{eq:Continuity_reachability_sets} holds. For any $\eps>0$ pick $\delta=\delta(\eps)$ as in \eqref{eq:Continuity_reachability_sets}. According to \eqref{eq:UGATT_wrt_Aeps} for any $R>0$ we have
\[
\|x\|_{\A} \leq R ,\ d\in \Dc ,\ t\geq \tau(R,\delta(\eps)) \srs \phi(t,x,d) \in \A_\delta
\]
and in view of \eqref{eq:Continuity_reachability_sets}:
\[
\|x\|_{\A} \leq R ,\ d\in \Dc ,\ t\geq \tau(R,\delta(\eps)) \srs \|\phi(t,x,d)\|_{P_+(\A)}\leq \eps.
\]
Since $\A \subset P_+(\A)$ it follows that $\|x\|_{\A} \leq \|x\|_{P_+(\A)}+C_1$, for a certain $C_1>0$. Denoting 
$\bar\tau(R,\eps):=\tau(R+C_1,\eps)$ we finally obtain
\[
\|x\|_{P_+(\A)} {\leq} R ,\, d\in \Dc ,\, t{\geq} \bar\tau(R,\delta(\eps)) \srs \|\phi(t,x,d)\|_{P_+(\A)}{\leq} \eps.
\]
This shows, that $P_+(\A)$ is UGATT. $P_+(\A)$ is bounded since $\A_\eps$ are bounded for any $\eps>0$. Invariance of $P_+(\A)$ follows from 
Lemma~\ref{lem:properties_A_epsilon_0}.
The first part of the proposition is proved.

Now assume that $\A$ is a robust invariant set. Assume that the claim of the proposition does not hold, and there exists a UGATT set $S$ which is not dense in $P_+(\A)$ (w.r.t. topology generated by $\|\cdot\|$) and so that $\A\subset S \subset P_+(\A)$.

Then for any $r,\eps >0$ there is a $\tau=\tau(r,\eps)$ satisfying
\begin{eqnarray}
\hspace{-8mm} \|x\|_{S} \leq r,\ d\in\Dc,\ t \geq \tau(r,\eps) \ \Rightarrow \ \|\phi(t,x,d)\|_{S} \leq \eps.
\label{eq:UGATT_proff}
\end{eqnarray}	

Since $S$ is not dense in $P_+(\A)$, there is some $y\in P_+(\A)$: $\|y\|_S>0$. Set $\eps:=\frac{\|y\|_S}{2}$ and $r:=1$.
Pick $h>\tau(1,\eps)$. 

As $\A$ is a robust invariant set, for these $\eps,h$ there is a $\delta >0$ (without loss of generality let $\delta<1$) so that \eqref{eq:RobEqPoint} holds. 
Since $\A \subset S$, $\|\phi(t,x,d)\|_S\leq \|\phi(t,x,d)\|_\A$ and thus
\begin{eqnarray}
\hspace{-8mm} t\in[0,h],\ \|x\|_{\A} \leq \delta,\ d\in\Dc \; \Rightarrow \;  \|\phi(t,x,d)\|_{S} \leq \eps.
\label{eq:RobEqPoint_proof}
\end{eqnarray}

As $y\in P_+(\A)$, then also $y\in \A_\delta = \Rc (B_\delta(\A))$.
Pick any $t_\delta>0$, $x_\delta \in B_\delta(\A)$, $d_\delta\in\Dc$ so that $\phi(t_\delta,x_\delta,d_\delta)=y$.

Applying \eqref{eq:RobEqPoint_proof} for $x:=x_\delta$ we see that 
\[
\|\phi(t,x_\delta,d_\delta)\|_{S} \leq \eps = \frac{\|y\|_S}{2}
\]
for $t\in[0,h]$, and in particular for $t\in[0,\tau(1,\eps)]$.

But $\|\phi(t_\delta,x_\delta,d_\delta)\|_S=\|y\|_S = 2\eps$ which means that $t_\delta>h>\tau(1,\eps)$.

On the other hand, $\A \subset S$ and thus $\|x_\delta\|_S \leq \|x_\delta\|_\A \leq \delta <1=r$ and according to \eqref{eq:UGATT_proff} it holds \mbox{$\|\phi(t_\delta,x_\delta,d_\delta)\|_{S} \leq \eps$}, a contradiction.
This finishes the proof.
\end{proof}

In the next example, we are going to show that the assumption of robust invariance of $\A$ is necessary 
for the second part of Proposition~\ref{prop:from_ULIM_to_UGATT_strengthening} to hold.
 
\begin{example}
\label{exam:UGATT_Robustness_Invariant_Sets}
Let $\Dc :=L^\infty(\R_+,\R)$, $x(t), y(t) \in\R$ and consider a planar system
\begin{subequations}
\begin{align}
\dot{x}(t) =&\ |d(t)| \big(1-x(t)\big) |y(t)| - x^3(t) - x^{1/3}(t), \\
\dot{y}(t) =&\ - y^3(t) - y^{1/3}(t).
\end{align}
\label{eq:Counterexample_Robust_Inv_Sets}
\end{subequations}
It is easy to see that \eqref{eq:Counterexample_Robust_Inv_Sets} is a forward complete system, $\A:=(0,0)$ is its equilibrium which is however not a robust equilibrium. All solutions of \eqref{eq:Counterexample_Robust_Inv_Sets} converge in finite time to $\A$ which makes $\A$ a uniform global attractor (more details you can find in \cite{MiW17a}). 
Moreover, \eqref{eq:Counterexample_Robust_Inv_Sets} is RFC since $\dot{x}(t) > 0$ whenever $x(t) < 0$ and $\dot{x}(t) < 0$ for $x(t) > 1$.

Pick $\eps<1$. The following inclusions hold:
\[
\Rc\big((-\sqrt{\eps},\sqrt{\eps})^2\big) \subset \Rc\big(B_\eps((0,0)^T)\big) \subset
\Rc\big((-\eps,\eps)^2\big).
\]
It is not hard to compute that $\Rc\big((-\eps,\eps)^2\big)= [-\eps,1) \times (-\eps,\eps)$,
which implies that $P_+(\A)=[0,1)\times\{0\}$.

Thus, $P_+(\A)$ is UGATT, but it contains a set $S=\A$ which is again UGATT and is not dense in $P_+(\A)$.
~ \hfill{} $\square$
\end{example}

\subsection{Relations of weak attractivity to other stability notions}

For systems without disturbances the concept of weak attractivity has been introduced in \cite{Bha66} and it was one of the basic notions upon which stability theory for systems without disturbances in \cite{BhS02} has been developed. 
On the other hand, in the language of \cite{SoW96}, weak attractivity is the limit property with zero
gain. 

As long as the author is concerned, the notion of uniform weak attractivity has been introduced in \cite{MiW17a, MiW17b},  but it was implicitly used in \cite[Proposition 2.2]{Tee13} and \cite[Proposition 3]{SuT16} to characterize UGAS of finite-dimensional hybrid and stochastic systems.
The concept of weak attractivity is intimately related to that of recurrent sets:
\begin{definition}{}
    \label{d:stability_new_2}
Consider a system $\Sigma=(X,\Dc,\phi)$. A set $\A \subset X$ is called
\begin{enumerate}[label=(\roman*)]

\item globally recurrent, if
\begin{eqnarray}
\hspace{-5mm} x\in X,\ d \in\Dc \srs \exists t\geq 0: \phi(t,x,d) \in \A.
\label{eq:Glob_recurr}
\end{eqnarray}

\item uniformly globally recurrent, if $\forall R>0$ $\exists \tau=\tau(R)$:
\begin{eqnarray}
\hspace{-15mm} \|x\|_\A \leq R,\ d \in\Dc \srs \exists t\in[0,\tau]: \phi(t,x,d) \in \A.
\label{eq:Unif_Glob_recurr}
\end{eqnarray}

\end{enumerate}
\end{definition}

The notion of a recurrent set is fairly classical, see e.g. \cite[Definition 1.1 in Chapter 3]{BhS02}.
Together with a more recent notion of uniformly recurrent sets (see e.g. \cite[Lemma 6.5]{Kha11}, \cite[Section 2.4]{Tee13}) it is widely used in the theory of hybrid and especially stochastic systems \cite{Kha11,Tee13,SuT16}.
The following relations between weak attractivity and recurrence are valid:	
\begin{remark}
\label{rem:Weak_attractivity_and_recurrence}
Let $\Sigma$ be a system and let $\A \subset X$ be arbitrary.
Then the following relations hold:
\vspace{-2mm}
\begin{itemize}
   \item If $\A$ is globally recurrent, then $\A$ is globally weakly attractive.	\vspace{-2mm}
   \item If $\A$ is uniformly globally recurrent, then $\A$ is uniformly globally weakly attractive.	\vspace{-2mm}
   \item $\A$ is globally weakly attractive $\Iff$ any neighborhood of $\A$ is globally recurrent.	\vspace{-2mm}
   \item $\A$ is uniformly globally weakly attractive $\Iff$ any neighborhood of $\A$ is uniformly globally recurrent.
\end{itemize}
\end{remark}

Finally, in some works the term of weak attractivity (or weak stability) is used with a meaning, which is quite different from that used in this paper.
In particular, in infinite-dimensional linear systems theory "weak" stability means convergence of solutions in the weak topology on $X$, see e.g. \cite{EFN07}. 
In \cite[Definition 4.1]{LiK12} weak attractivity has been used to study convergence properties of stochastic systems with small parameters. 
In \cite[Definition 9.2]{Rox65} weak stability of systems with multiple solutions is understood in the sense that for each neighborhood $W_1$ of $0$ there is a neighborhood $W_2$ of 0, so that for each $x\in W_2$ there is a solution staying within $W_1$ for all times, and analogously weak attractivity could be defined. 
Yet another notion hiding behind the name "weak attractivity" has been used in the context of dynamical systems under perturbations in \cite{Gru02b}.

\section{Characterizations of UGAS}

Arguments, exploited in the previous section to characterize pUGAS, can be adapted to obtain criteria of UGAS.

The relation between uniform global attractivity and uniform global weak attractivity is given by:
\begin{proposition}
\label{prop:UGATT_and_Un_approachability_v2}
Consider a forward complete system $\Sigma=(X,\Dc,\phi)$.
Then $\A$ is a UGATT robustly invariant set if and only if  $\A$ is uniformly stable and uniformly globally weakly attractive.
\end{proposition}

\begin{proof}
$\Rightarrow$. 
Let $\Sigma=(X,\Dc,\phi)$ be a forward complete system.
Due to items \ref{def:UniformGlobalWeakAttractivity} and \ref{def:UniformGlobalAttractivity}
of Definition~\ref{d:stability_new}, UGATT of $\A$ implies uniform global weak attractivity of $\A$. To show uniform
 stability of $\A$ fix $\eps>0$ and
$r>0$. By item\,\ref{def:UniformGlobalAttractivity} of Definition~\ref{d:stability_new} there is a $\tau=\tau(\eps,r)>0$ s.t.
\begin{eqnarray*}
\hspace{-5mm}\|x\|_\A \leq r,\ d\in\Dc,\ t \geq \tau(r,\eps) \srs \|\phi(t,x,d)\|_\A \leq \eps.
\end{eqnarray*}
Since $\A$ is a robust invariant set of $\Sigma$, there is a $\hat\delta>0$ corresponding to $\tau$ so that 
\[
t\in[0,\tau],\ \|x\|_\A \leq \hat\delta, \ d \in \Dc \srs  \|\phi(t,x,d)\|_\A \leq \eps.
\]
The combination of the two implications shows that $\A$ is uniformly stable with $\delta:=\min\{r,\hat\delta\}$.

$\Leftarrow$. Fix $\eps>0$. Since $\A$ is uniformly stable, there is a $\delta>0$ so that
\eqref{eq:Uniform_Stability} holds. This shows that $\A$ is a robust invariant set.
Now pick any $r>0$. Since $\A$ is a uniform global weak attractor, for the
above $\delta$ there exists a $\tau=\tau(\delta,r)$ so that 
\begin{eqnarray*}
\hspace{-9mm}\|x\|_\A \leq r,\ d\in\Dc \srs \exists \bar{t}(x,d,\delta) \leq \tau:\ \|\phi(t,x,d)\|_\A \leq \delta.
\end{eqnarray*}
Finally, ULS of $\Sigma$ and the cocycle property imply that
\[
\|x\|_\A \leq r,\ d\in\Dc,\ t \geq \bar{t}(x,d,\delta) \srs\|\phi(t,x,d)\|_\A \leq \eps.
\]
Specifying this to $t\geq \tau(\eps,r)$, we see that $\A$ is UGATT.
\end{proof}


Now we are able to characterize UGAS property.
\begin{theorem}
\label{cor:ULIM_RFC_LS_v2}
Consider a forward complete system $\Sigma=(X,\Dc,\phi)$. Let $\A \subset X$ be a bounded set. 
The following statements are equivalent:
\begin{enumerate}[label=(\roman*)]
 \item \label{enum:UGAS} $\A$ is UGAS.
 \item \label{enum:RFC_UGATT_RIS} $\Sigma$ is RFC and $\A$ is a UGATT robust invariant set.
 \item \label{enum:RFC_ULS_UGWATT} $\Sigma$ is RFC and $\A$ is uniformly stable and uniformly globally weakly attractive.
 \item \label{enum:UGS_UGWATT} $\A$ is uniformly globally stable and uniformly globally weakly attractive.
 \item \label{enum:UGS_UGATT} $\A$ is uniformly globally stable and UGATT.
\end{enumerate}
\end{theorem}

\begin{proof}
\ref{enum:UGAS} $\Rightarrow$ \ref{enum:RFC_UGATT_RIS}. Estimate $\beta(\|x\|_\A,t)\leq \beta(\|x\|_\A,0)$ shows global stability of $\A$ and hence RFC of a system $\Sigma$ and robust invariance of $\A$. UGATT of $\A$ holds with $\tau(r,\eps)$ as a solution of the equation $\beta(r,\tau)=\eps$ provided such a solution exists or $\tau(r,\eps):=0$ if this equation is not solvable.

\ref{enum:RFC_UGATT_RIS} $\Iff$ \ref{enum:RFC_ULS_UGWATT}. Follows due to Proposition~\ref{prop:UGATT_and_Un_approachability_v2}.

\ref{enum:RFC_ULS_UGWATT} $\Rightarrow$ \ref{enum:UGS_UGWATT}. Lemma~\ref{lem:Weak_attr_plus_RFC_imply_pUGS} shows that $\A$ is Lagrange stable. Since $\A$ is ULS, Remark~\ref{rem:UGS_equals_LagrStab_and_ULS} ensures UGS of $\A$.

\ref{enum:UGS_UGWATT} $\Rightarrow$ \ref{enum:UGS_UGATT}. Since $\A$ is UGS, $\Sigma$ is RFC. Now already proved equivalence \ref{enum:RFC_UGATT_RIS} $\Iff$ \ref{enum:RFC_ULS_UGWATT} shows the claim.

\ref{enum:UGS_UGATT} $\Rightarrow$ \ref{enum:UGAS}. Follows by item (ii) of Lemma~\ref{lem:UGATT_plus_LagrangeSt_imply_pUGAS}.
\end{proof}

\section{Weak attractivity vs uniform weak attractivity}

We devote this final section to the comparison between weak attractivity and uniform weak attractivity.
This can be shown by looking at the special classes of systems: linear systems over Banach spaces and ordinary differential equations. As we will see, in the first case uniform weak attractivity is much stronger than weak attractivity and is equivalent to UGAS of $\{0\}$. In the ODE case uniform and non-uniform weak attractivity coincide provided the disturbances are uniformly bounded. In both cases, criteria for UGAS and pUGAS become much simpler.

\subsection{Linear systems}
\label{sec:linear_systems}

Let us take a quick look at linear systems over Banach spaces without disturbances of the form 
\begin{eqnarray}
\dot{x} = Ax,
\label{eq:LinSys}
\end{eqnarray}
where $A:D(A)\to X$ is the infinitesimal generator of a $C_0$-semigroup $T(\cdot)$ of bounded operators over a Banach space $X$ with a domain of definition $D(A)$.
Mild solutions of \eqref{eq:LinSys} are given by $\phi(t,x)=T(t)x$, for any $x\in X$ and $t\geq 0$.
In this way, \eqref{eq:LinSys} gives rise to a system without disturbances.

For such systems criteria from Theorem~\ref{cor:ULIM_RFC_LS_v2} take a particularly simple form.

\begin{proposition}
\label{prop:Uniform_appr_equals_Exp_Stability_for_linear_systems_v2}
The following statements are equivalent for the system \eqref{eq:LinSys}:
\begin{enumerate}[label=(\roman*)]
  \item \label{item:ExpSt} $T$ is exponentially stable.
	\item \label{item:LinUGAS} $\{0\}$ is UGAS. 
	\item \label{item:LinUGATT} $\{0\}$ is uniformly globally attractive.
	\item \label{item:LinUGWATT} $\{0\}$ is uniformly globally weakly attractive.
	\item \label{item:Lin_noncoerc} There is a non-coercive Lyapunov function for \eqref{eq:LinSys}.
\end{enumerate}

\end{proposition}

\begin{proof}
Equivalence  \ref{item:ExpSt} $\Iff$ \ref{item:LinUGAS} is easy to show \cite[Lemma 1]{DaM13}.
Implications \ref{item:LinUGAS} $\Rightarrow$ \ref{item:LinUGATT} $\Rightarrow$ \ref{item:LinUGWATT} holds already for general nonlinear systems.

Let us show that  \ref{item:LinUGWATT} $\Rightarrow$ \ref{item:LinUGAS}. It is well-known that every strongly continuous semigroup satisfies a bound of the type $\|T(t)x\|\leq M e^{\omega t}\|x\|$, for all $t\geq 0, x\in X$ and suitable $M,\omega \in \R$,
    \cite[Theorem 2.1.6]{CuZ95}. Hence it is RFC.  By
    Lemma~\ref{lem:Weak_attr_plus_RFC_imply_pUGS}, $T$ is
    Lagrange stable. In
    particular, for any $x\in X$ it holds that $\sup_{t\geq0}\|T(t)x\| <
    \infty$ and the Banach-Steinhaus theorem implies
		$\sup_{t\geq 0,\, \|x\|=1}\|T(t)x\| < \infty$
    which means that
    $\{0\}$ is globally stable.
		Thus item \ref{enum:UGS_UGWATT} of Theorem~\ref{cor:ULIM_RFC_LS_v2} shows the claim.

Equivalence \ref{item:LinUGAS} $\Iff$ \ref{item:Lin_noncoerc} follows by Proposition~\ref{t:noncoeLT}.
\end{proof}
\begin{remark}
\label{rem:Approachability_weaker_than_Uniform_Approachability_v2}
Proposition~\ref{prop:Uniform_appr_equals_Exp_Stability_for_linear_systems_v2}
shows that uniform weak attractivity is a stronger
requirement than weak attractivity. In fact, every strongly stable
$C_0$-semigroup $T$ (a semigroup, satisfying $T(t)x\to 0$ as $t\to\infty$
for any $x\in X$) is globally weakly attractive, while it is not necessarily exponentially stable (and hence does not have to be uniformly weakly attractive).
\end{remark}

\subsection{Ordinary differential equations}

Here we specify the developments of Section~\ref{sec:pUGAS_equals_ULIM} to the case of ODEs
\begin{eqnarray}
\dot{x}=f(x,d),
\label{eq:ODE_Sys}
\end{eqnarray} 
where $f:\R^n\times D \to \R^n$ is locally Lipschitz continuous w.r.t. the first argument uniformly w.r.t. the second one, $D$ is a compact subset of $\R^m$
and disturbances $d$ belong to the set $\Dc:=L_\infty(\R_+,D)$ of Lebesgue measurable globally essentially bounded functions with values in $D$.

In contrast to
Remark~\ref{rem:Approachability_weaker_than_Uniform_Approachability_v2},
for systems of nonlinear ordinary differential equations with uniformly bounded
disturbances the notions of global weak attractivity and uniform global
weak attractivity coincide which is a simple application of \cite[Corollary III.3]{SoW96}:
\begin{proposition}
\label{prop:Approachability_equals_Uniform_Approachability}
Consider a system \eqref{eq:ODE_Sys} with $\Dc$ as above. Let $\A \subset \R^n$ be any bounded set. Then $\A$ is uniformly globally weakly attractive if and only if it is globally weakly attractive.
\end{proposition}

\begin{proof}
Pick any $r,\eps>0$.
An application of \cite[Corollary III.3]{SoW96} with $C:=\overline{B_r(\A)}$, $\Omega:=B_{\eps}(\A)$,
$K:=\overline{B_{\frac{\eps}{2}}(\A)}$ ensures existence of a $\tau(r,\eps)$ such that for all $x\in C$ and
any $d\in \Dc$ there is a $t\leq\tau$ with $|\phi(t,x,d)| <\eps$.
\end{proof}
%

\begin{remark}
\label{rem:ISS_version_of_LIM_equals_ULIM}
A counterpart of Proposition~\ref{prop:Approachability_equals_Uniform_Approachability} 
can be also proved in the context of input-to-state stability theory, see \cite[Proposition 9]{MiW17b}.
\end{remark}

\begin{remark}
\label{rem:Uniform_Boundedness_of_disturbances_is_essential}
It is easy to see that Proposition~\ref{prop:Approachability_equals_Uniform_Approachability} is not valid for 
systems with disturbances which are not uniformly bounded. 
Indeed, the system $\dot{x} = - \frac{1}{|d| + 1}x$ with $\Dc= L_\infty(\R_+,\R)$ has $0$ as
a global weak attractor but not as a uniform one.
\end{remark}

The criterion for practical UGAS of a system \eqref{eq:ODE_Sys} takes a particularly simple form.
We have the following corollary of Proposition~\ref{prop:Approachability_equals_Uniform_Approachability}:
\begin{corollary}
\label{cor:ODE_pUGAS}
\eqref{eq:ODE_Sys} is pUGAS if and only if there is a bounded globally weakly attractive set $\A\subset \R^n$ for \eqref{eq:ODE_Sys}.
\end{corollary}

\begin{proof}
Follows from Theorem~\ref{thm:RFC_weak_attr_imply_pUGAS}, Proposition~\ref{prop:Approachability_equals_Uniform_Approachability} and 
\cite[Proposition 5.1]{LSW96} (showing that forward completeness and robust forward completeness are equivalent notions for ODEs with uniformly bounded disturbances).
\end{proof}


\vspace{-7mm}

\section{Conclusions}

For robustly forward complete systems we have shown that practical uniform global asymptotic stability is equivalent to the existence of a global uniform weak attractor.
For the same class of systems, we proved that existence of a non-coercive Lyapunov function implies practical UGAS.
Relations between weak attractivity and uniform weak attractivity have been discussed for the ODEs and linear infinite-dimensional systems without disturbances.

\vspace{-4mm}

\section*{Acknowledgements}

This research has been supported by the German Research Foundation (DFG) within the project "Input-to-state stability and stabilization of distributed parameter systems" (grant Wi 1458/13-1).
The author thanks Fabian R. Wirth for fruitful conversations and his feedback concerning this paper.
The author is also thankful to the anonymous reviewers for the careful
evaluation of the paper and their valuable suggestions.


\bibliographystyle{abbrv}

\bibliography{C:/Users/Andrii/GoogleDrive/TEX_Data/Mir_LitList}

\end{document}